\documentclass[11pt,a4paper]{amsart}
\usepackage{amsmath,amsfonts,amsthm,amssymb}
\usepackage[T1]{fontenc}
\usepackage{color}
\usepackage{epsfig}
\usepackage{graphicx}
\usepackage{hyperref}
\usepackage[capitalise]{cleveref}
\usepackage{xspace}
\usepackage{enumitem}
\usepackage[left=1in,right=1in,top=1in,bottom=1in,foot=0.5in]{geometry}
\usepackage{thm-restate}
\usepackage{soul}
\usepackage[square,sort,comma,numbers]{natbib}
\usepackage{etoolbox}

\usepackage{tikz}
\usepackage{tkz-euclide}
\usetikzlibrary{positioning,calc}
\usetikzlibrary{decorations.pathmorphing, arrows.meta}
\usetikzlibrary{shapes.geometric}
\tikzset{harp/.style={arrows = {[harpoon]<->[harpoon]}}} % left right edge
\usepackage{color} % custom colors
\usepackage{standalone}

\usepackage{subcaption}

\usepackage{amsthm}
\newtheorem{theorem}{Theorem}

\newtheorem{lemma}{Lemma}
\newtheorem{corollary}{Corollary}

\theoremstyle{definition}
\newtheorem{definition}{Definition}

\AtBeginEnvironment{proof}{\setcounter{case}{0}}%make sure the case numbers are reset within proofs
\newtheorem{case-int}{Case}
\newtheorem{case-pos}{Case}
\newtheorem{case-kk}{Case}
\newtheorem{case-kk+1}{Case}
\newtheorem{case-L4}{Case}
\newtheorem{case-L5}{Case}
\newtheorem{case-L6}{Case}
\newtheorem{case-int1}{Case}
\newtheorem{case-int2}{Case}
\newtheorem{case-int3}{Case}
\newtheorem{subcase-kkk}{Subcase}
\newtheorem{subcase-kkk+1}{Subcase}
\newtheorem{subcase-x|y}{Subcase}
\newtheorem{subcase-x|k}{Subcase}
\newtheorem{subcase-ext}{Subcase}
\newtheorem{subsubcase-x|k}{Subsubcase}
\newtheorem{subsubcase-k|x}{Subsubcase}
\newtheorem{claim-swm}{Claim}
\newtheorem{claim-sswm}{Claim}

\numberwithin{equation}{section}

\DeclareMathOperator{\conv}{conv}

\newcommand{\coloneqq}{\mathrel{:=}}

\title{Smooth weakly modular graphs}

\author[V.\ Chepoi, B.J.\ Schmidt, and P.F.\ Stadler]{Victor
  Chepoi$^{1,2}$, Bruno J.\ Schmidt$^{3,4}$, and
  Peter F.\ Stadler$^{3,4,5,6,7}$}

\begin{document}

\maketitle

\vspace*{-0.2cm}
%	\medskip
\centerline{$^{1}$LIS, Aix-Marseille Universit\'e, CNRS, and Universit\'e
  de Toulon} \centerline{Facult{\'e} des Sciences de Luminy, F-13288
  Marseille Cedex 9, France} \centerline{\textsf{victor.chepoi@lis-lab.fr}}
\medskip
\centerline{$^{2}$ Institut Universitaire de France (IUF)}

\medskip
\centerline{$^{3}$ Max Planck Institute for Mathematics in the Sciences,}
\centerline{Inselstra{\ss}e 22, D-04103 Leipzig, Germany}

\medskip
\centerline{$^{4}$ Bioinformatics Group, Department of Computer Science \&
  Interdisciplinary Center}
\centerline{for Bioinformatics, Leipzig University, H{\"a}rtelstraße
  16--18, D-04107 Leipzig, Germany}
\centerline{\sf{bruno@bioinf.uni-leipzig.de}}

\medskip
\centerline{$^{5}$ Department of Theoretical Chemistry, University of Vienna,}
\centerline{W{\"a}hringerstra{\ss}e 17, A-1090 Wien, Austria}

\medskip \centerline{$^{6}$ Facultad de Ciencias, Universidad National de
  Colombia, Bogot{\'a}, Colombia}

\medskip \centerline{$^7$ Santa Fe Institute, 1399 Hyde Park Rd., Santa Fe,
  NM 87501, USA} \centerline{\textsf{studla@bioinf.uni-leipzig.de}}

% \iffalse

%   \author{Victor Chepoi} \address[V.~Chepoi]{Aix-Marseille Universit\'e
%   and CNRS, LIS, Marseille, France}

% \address{Institut Universitaire de France}
% \email{victor.chepoi@lis-lab.fr}

% \author{Bruno J. Schmidt}
% \address[B.~Schmidt]{Max Planck Institute for Mathematics in the Sciences, Inselstraße 22,
% D-04103 Leipzig, Germany}
% \address{Bioinformatics Group, Department of Computer Science \&
% Interdisciplinary Center for Bioinformatics, Leipzig University,
% H\"artelstraße 16–18, D-04107 Leipzig, Germany}
% \email{bruno@bioinf.uni-leipzig.de}

% \author{Peter Stadler}
% \address[P.~Stadler]{Max Planck Institute for Mathematics in the Sciences, Inselstraße 22,
% D-04103 Leipzig, Germany}
% \address{Bioinformatics Group, Department of Computer Science \&
% Interdisciplinary Center for Bioinformatics, Leipzig University,
% H\"artelstraße 16–18, D-04107 Leipzig, Germany}
% \address{Department of Theoretical Chemistry, University of Vienna,
% W\"ahringerstraße 17, A-1090 Wien, Austria}
% \address{Facultad de Ciencias, Universidad National de Colombia, Bogot\'{a},
% Colombia}
% \address{Santa Fe Institute, 1399 Hyde Park Rd., Santa Fe, NM 87501, USA}
% \email{studla@bioinf.uni-leipzig.de}
% \fi

%\begin{document}

\begin{abstract}
  A graph $G=(V,E)$ is called \emph{smooth} (respectively, \emph{strongly
    smooth}) if for any two vertices $u,v\in V$, the distance point-shadow
  $v|u \coloneqq \{ x\in V: d(u,x)=d(u,v)+d(v,x)\}$, respectively, the
  point-shadow $v/u \coloneqq \{ x\in V: v\in \conv(u,x)\}$, is
  geodesically convex.  Smooth graphs have been introduced by Nebesk{\'y}
  (2005) in the context of step systems. Graphs with convex point-shadows
  and convex distance point-shadows also naturally occur in convexity
  theory. Bre\v{s}ar et al. (2026) recently showed that several classes of
  graphs are smooth and that smoothness is preserved by Cartesian products,
  gated amalgams, and isometric subgraphs. Weakly modular graphs comprise
  the most important classes of graphs from Metric Graph Theory: median,
  modular, Helly, bridged, and dual polar graphs.
  
  In this note, we characterize smooth and strongly smooth weakly modular
  graphs in terms of forbidden isometric subgraphs on 5 and 7
  vertices. This settles Problem 1 of the paper by Bre\v{s}ar et al. We
  also characterize prime strongly smooth weakly modular graphs, i.e.,
  strongly smooth weakly modular graphs that cannot be obtained from
  smaller graphs by Cartesian products and gated amalgams.

\end{abstract}

\maketitle

\section{Introduction}

Smooth graphs have been introduced by Nebesk\'y \cite{Ne} in the context of
step systems in graphs \cite{Ne_step}. They have been studied in depth by
Bre\v{s}ar, Changat, Narasimha-Shenoi, Schmidt, and Stadler in
\cite{BrChNSScSt}, where several classes of graphs have been shown to be
smooth. Geometrically, \emph{smooth graphs} can be defined as the graphs
$G=(V,E)$ in which the distance point-shadows
$v|u \coloneqq \{ x\in V: v\in I(u,x)\}$ are convex sets of $G$, where
$I(p,q) \coloneqq \{ w\in V: d(p,w)+d(w,q)=d(p,q)\}$ denotes geodesic
intervals.  A related (but different) notion of shadow
$u/v\coloneqq \{ x\in V: v\in \conv(u,x)\}$ -- and, more generally,
$A/B=\{ x\in V: \conv(B\cup \{ x\})\cap A\ne\varnothing\}$ -- concerns
separation by halfspaces in convexity theory. Here $\conv(A)$ denotes the
convex hull of the vertex set $A$.  The convexity of shadows of the form
$A/B$ or $x/B$ characterizes the $S_4$- and $S_3$-separability by
halfspaces in abstract convexity spaces \cite{Ch_domain,ChS4,ChS3} and in
graphs \cite{Ch_bipartite,ChS4,ChS3} (see \cite{vdV} for a presentation of
some of those results).  Then replacing distance point-shadows $v|u$ by
point shadows $v/u$ leads to a notion of \emph{strongly smooth} graphs.
Weakly modular graphs have been introduced in \cite{BaCh_helly,Ch_metric}
as the graphs satisfying two metric conditions: the triangle and the
quadrangle conditions. They provide a common generalization of many classes
of graphs defined by distance properties, in particular, median, modular,
Helly, bridged, and dual polar graphs; see \cite{BaCh_survey} for a survey
and \cite{CCHO} for a unifying theory. The main result of this note is a
characterization of smooth and strongly smooth weakly modular graphs in
terms of forbidden isometric subgraphs on 5 and 7 vertices, answering
Problem 1 of \cite{BrChNSScSt}.
  
Bridged graphs are the graphs satisfying one of the basic properties of
convexity in Euclidean spaces: the neighborhoods around convex sets are
convex. Bridged graphs are exactly the graphs in which all isometric cycles
have length 3 \cite{FaJa,SoCh}. Topologically, they are characterized as
the graphs whose clique complexes are simply connected and do not contain
induced 4-cycles and 5-cycles \cite{Ch_CAT}. Such simplicial complexes have
been rediscovered in \cite{JaSw} and called systolic complexes. Bridged
graphs and systolic complexes have been intensively studied in geometric
group theory, where they are considered as simplicial complexes with
combinatorial nonpositive curvature. With this motivation, weakly systolic
complexes and their 1-skeleta---weakly bridged graphs---have been
introduced in \cite{Os}.  They are exactly the weakly modular graphs with
convex balls, or, equivalently, those without induced 4-cycles \cite{ChOs}.
We will show below that smoothness and strong smoothness are equivalent for
weakly bridged graphs; moreover, we shall describe a characterization in
terms of two forbidden isometric subgraphs.

A method of studying weakly modular graphs is to investigate how they are
obtained from simpler graphs by Cartesian products and gated
amalgamations. Given a class of weakly modular graphs closed by these
operations, a central problem is to characterize the prime graphs in the
sense that are neither a Cartesian product nor a gated amalgamation of
smaller graphs in the same class.  The prime graphs are known for median
graphs \cite{Is}, quasi-median graphs \cite{BaMuWi}, weakly median graphs
\cite{BaCh_weak}, and bucolic graphs \cite{BrChChGoOs}, all of which are
free of induced $K_{2,3}$ and $W^-_4$ graphs. The weakly modular graphs
free of induced $K_{2,3}$ and $W^-_4$ are the \emph{pre-median
  graphs}. Every pre-median graph $G$ isometrically embeds into the
Cartesian product of its primes and, under some conditions on the primes,
$G$ is even a retract of this product \cite{Cha1,Cha2}. Chalopin et al.\
\cite{CCHO} characterized prime pre-median graphs as the 2-connected
pre-median graphs whose clique complexes are simply connected. Combining
this result with our characterization we obtain a characterization of
primes of strongly smooth weakly modular graphs.

\section{Preliminaries}

\subsection{Basic notions and notations} All graphs $G=(V,E)$ considered in
this note are undirected, connected, and contain neither multiple edges,
nor loops but are not necessarily finite. For two vertices $v,w\in V$ we
write $v\sim w$ when there is an edge connecting $v$ with $w$ and
$v\nsim w$ otherwise.  The subgraph of $G$ \emph{induced by} a set
$A\subseteq V$ is the graph $G[A]=(A,E')$ such that $uv\in E'$ if and only
if $u,v\in A$ and $uv\in E$. A \emph{square} (respectively,
\emph{triangle}) is an induced 4-cycle $v_1v_2v_3v_4$ (respectively, a
3-cycle $v_1v_2v_3$) of $G$. The \emph{length} of a $(u,v)$-path $P$ is the
number of edges on $P$. A \emph{$(u,v)$--geodesic} (or a
\emph{$(u,v)$--shortest path}) is an $(u,v)$-path of smallest length. The
\emph{distance} $d(u,v)=d_G(u,v)$ between $u$ and $v$ is the length of a
$(u,v)$--geodesic. An \emph{isometric embedding} of a graph $H=(A,F)$ into
a graph $G=(V,E)$ is a map $f:A \rightarrow V$ such that
$d_G(f(x),f(y))=d_H(x,y)$ for any $x,y\in A$. The image of $H$ under an
isometric embedding is called an \emph{isometric subgraph} of $G$. Each
isometric subgraph is an induced subgraph.

The \emph{interval} $I(u,v)$ between $u$ and $v$ consists of all vertices
on $(u,v)$--geodesics, that is, of all vertices \emph{between} $u$ and $v$:
$I(u,v)\coloneqq \{x\in V: d(u,x)+d(x,v)=d(u,v)\}$. An induced subgraph of
$G$ (or the corresponding vertex-set $A$) is called \emph{convex} (or
\emph{geodesically convex}) if it includes the interval of $G$ between any
pair of its vertices. By $\conv(S)$ we denote the \emph{convex hull} of $S$
which is the minimal convex set that contains $S$. A set of vertices $A$ is
called \emph{locally convex} if the subgraph $H=G[A]$ induced by $A$ is
connected and for any $x,y\in A$ with $d_H(x,y)=2$ the interval $I(x,y)$ is
included in $A$.  We say that a graph $G$ is a \emph{graph with convex
  intervals} if for any two vertices $u,v$, the interval $I(u,v)$ is a
convex set of $G$.

For a vertex $s$ of $G$ and an integer $r\ge 1$, we will denote by
$B_r(s)\coloneqq \{ x\in V: d(s,x)\le r\}$ the \emph{ball} in $G$ of radius
$r$ centered at $s$, and the subgraph induced by this ball. More generally,
the $r$--\emph{ball around a set} $S\subseteq V$ is the set (or the
subgraph induced by) $B_r(S)\coloneqq \{ v\in V: d(v,S)\le r\},$ where
$d(v,S)=\min \{ d(v,s): s\in S\}$.  Similarly, the \emph{sphere} centered
at $s$ of radius $r$ is the set (or the subgraph induced by)
$S_r(s)\coloneqq \{ x \in V: d(s,x) = r\}$.  As usual,
$N(s)\coloneqq S_1(s)$ denotes the set of neighbors of a vertex $s$ and
$N[s]=B_1(s)$.

An induced subgraph $H$ of a graph $G$ is \emph{gated} \cite{DrSch} if for
every vertex $x$ outside $H$ there exists a vertex $x'$ in $H$ (the
\emph{gate} of $x$) such that $x'\in I(x,y)$ for any $y$ of $H$.  Gated
sets are convex. A graph $G$ is a \emph{gated amalgamation} of graphs $G_1$
and $G_2$ if $G_1$ and $G_2$ are (isomorphic to) two intersecting gated
subgraphs of $G$ whose union is $G.$ A graph $G$ with at least two vertices
is called \emph{prime} if it is neither a Cartesian product nor a gated
amalgamation of smaller graphs.

\subsection{Smooth and strongly smooth graphs}

For two vertices $u,v$ of a graph $G$, the \emph{point-shadow} \cite{ChS4}
of $v$ with respect to $u$ is the set
$v/u\coloneqq \{x\in V : v\in \conv(u,x)\}$.  Analogously, the
\emph{distance point-shadow} of $v$ with respect to $u$ is the set
$v|u\coloneqq \{x\in V: v\in I(u,x)\}$. In general, $v|u\subseteq v/u$ and
$v|u\ne v/u$.  We continue with the main definitions of this note:
\begin{definition}[Smooth graphs]
  A graph $G=(V,E)$ is called \emph{smooth} \cite{BrChNSScSt} if $G$
  satisfies the following equivalent conditions:
  \begin{enumerate}
  \item[(1)] for all vertices $u,v,w,x,y\in V$ such that
    $v\in I(u,x)\cap I(u,y)$ and $w\in I(x,y)$, we have $v\in I(u,w)$.
  \item[(2)] for all vertices $u,v\in V$, the distance point-shadow $v|u$
    is convex.
  \item[(3)] for any two adjacent vertices $u,v$, the distance point-shadow
    $v|u$ is convex.
  \end{enumerate}
\end{definition}
Property (2) suggests to also consider an analogous notion using the
point-shadow instead of the distance point-shadow:
\begin{definition}[Strongly smooth graphs]
  A graph $G$ is called \emph{strongly smooth} if for all vertices
  $u,v\in V$, the  point-shadow $v/u$  is convex.
\end{definition}

The term ``strongly smooth graph'' is justified by the following result
(see also \cite[Lemma 13]{BrChNSScSt} for a different proof):
\begin{lemma}\label{strongly-smooth->smooth} If $G$ is a strongly smooth
  graph, then $G$ is a graph with convex intervals. Consequently, strongly
  smooth graphs are smooth.
\end{lemma}
\begin{proof} Let $G$ be a strongly smooth graph and suppose that $G$
  contains a non-convex interval. Let $I(u,v)$ be a non-convex interval of
  $G$ with minimal $k=d(u,v)$. Since $I(u,v)$ is not convex, there exist
  two vertices $x,y\in I(u,v)$ and a vertex $z\in I(x,y)$ adjacent to $x$
  and not belonging to $I(u,v)$.  Then clearly $x$ is different from each
  of the vertices $u,v$. Furthermore, $d(z,u)\ge d(x,u)$ and
  $d(z,v)\ge d(x,v)$ and at least one of these inequalities, say the
  second, is strict. Then $d(z,v)>d(x,v)$ and $z\sim x$, implies that
  $x\in I(z,v)$.  Since $z\in \conv(u,v)$, the vertex $v$ belongs to the
  shadow $z/u$. Since $z\in z/u$ and $x\in I(z,v)$, from convexity of the
  shadow $z/u$ we conclude that $x\in z/u$. This implies that
  $z\in \conv(x,u)$. From the minimality choice of $d(u,v)$ it follows that
  the interval $I(u,x)$ is convex, whence $\conv(x,u)=I(x,u)$. Since
  $x\in I(z,u)$, necessarily $z\notin I(x,u)$, yielding $x\notin z/u$, a
  contradiction. This establishes that strongly smooth graphs have convex
  intervals.  Consequently, in strongly smooth graphs, the point-shadows
  coincide with distance point-shadows. Strongly smooth graphs therefore
  are smooth.
\end{proof}
As noted in \cite{BrChNSScSt}, the graph $W_4^-$ in Fig.~\ref{f:W4-} is
smooth.  However, $W_4^-$ is not strongly smooth, since the shadow $v/u$ is
not convex.

\begin{figure}[ht]
  \centering

  \begin{subfigure}[t]{0.3\textwidth}
    \centering
    \resizebox{\linewidth}{!}{
      \begin{tikzpicture}[every node/.style={circle, draw, fill=white, inner sep=1.5pt}]
        \node (u) at (-1,1.5) {$u$};
        \node (v) at (1,1.5)  {$v$};
        \node (x) at (1,3.5)  {$x$};
        \node (y) at (1,-0.5) {$y$};
        \node (w) at (3,1.5)  {$w$};
        \draw (u) -- (x) -- (w) -- (y) -- (u);
        \draw (u) -- (v) -- (x);
        \draw (v) -- (y);
      \end{tikzpicture}}
    \caption{}
  \end{subfigure}%
  \caption{The graph $W_4^-$.}
  \label{f:W4-}
\end{figure}
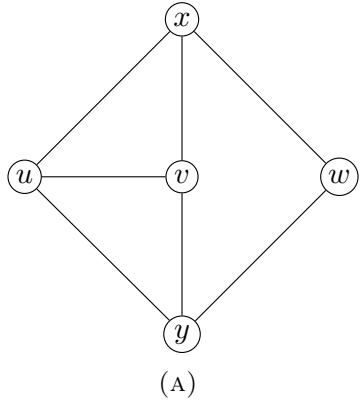

By definition of shadows, $v\in v|u$. Moreover, $x\in v|u$ and $y\in I(v,x)$
  implies $y\in v|u\subseteq v/u$. As an immediate consequence we have: 
\begin{lemma}
  \label{shadow-connected}
  For any two vertices $u,v$ of a graph $G$, both the shadow $v/u$ and the
  distance point-shadow $v|u$ induce nonempty connected subgraphs of $G$.
\end{lemma}

\subsection{Weakly modular graphs} We continue with the
definition of weakly modular graphs:
\begin{definition} [Weakly modular graphs]
  A graph $G=(V,E)$ is called \emph{weakly modular}
  \cite{BaCh_helly,Ch_metric} if $G$ satisfies the following Triangle and
  Quadrangle Conditions:
\begin{enumerate}
\item[(TC)] Triangle Condition: for any  $u,v,s \in V$ with
  $d(u,v)=1$ and $d(u,s)=d(v,s)$, there exists a common neighbor $x$
  of $u$ and $v$ such that $d(s,x)=d(s,u)-1$.
\item[(QC)] Quadrangle Condition: for any $u,v,w,s\in V$ with
  $d(u,w)=d(v,w)=1$, $d(u,v) = 2$, and $d(s,u)=d(s,v)=d(s,w)-1,$ there
  exists a common neighbor $x$ of $u$ and $v$ such that
  $d(s,x)=d(s,u)-1$.
\end{enumerate}
\end{definition}

An important property of weakly modular graphs is that local convexity
implies global convexity (for a generalization of this result to the larger
class of meshed graphs, see \cite{ChS3}).
\begin{lemma} \upshape{\cite{Ch_metric,ChS4}}
  \label{thm:local-convex->global-convex}
  A connected subgraph $H$ of a weakly modular graph $G$ is convex if and
  only if $H$ is locally convex.
\end{lemma}

Three vertices $v_1,v_2,v_3$ of a graph $G$ form a \emph{metric triangle}
\cite{Ch_delta} $v_1v_2v_3$ if the intervals $I(v_1,v_2), I(v_2,v_3),$ and
$I(v_3,v_1)$ pairwise intersect only in the common end-vertices, i.e.,
$I(v_i, v_j) \cap I(v_i,v_k)= \{v_i\}$ for any $1 \leq i, j, k \leq 3$.
If $d(v_1,v_2)=d(v_2,v_3)=d(v_3,v_1)=k,$ then this metric triangle is
called {\it equilateral} of {\it size} $k.$ An equilateral metric triangle
$v_1v_2v_3$ of size $k$ is called \emph{strongly equilateral} if
$d(v_1,v)=k$ for all $v\in I(v_2,v_3)$.

A metric triangle $v_1v_2v_3$ of $G$ is a \emph{quasi-median} of the
triplet $x,y,z$ if the following hold: % metric equalities are satisfied:
\begin{align*}
d(x,y)&=d(x,v_1)+d(v_1,v_2)+d(v_2,y),\\
d(y,z)&=d(y,v_2)+d(v_2,v_3)+d(v_3,z),\\
d(z,x)&=d(z,v_3)+d(v_3,v_1)+d(v_1,x).
\end{align*}
If $v_1$, $v_2$, and $v_3$ are the same vertex $v$, or equivalently, if the
size of $v_1v_2v_3$ is zero, then this vertex $v$ is called a \emph{median}
of $x,y,z$.  While a median may not exist and may not be unique, a
quasi-median of every triplet $x,y,z$ always exists.

Weakly modular graphs can be characterized in the following way in terms of
metric triangles:
\begin{lemma} \upshape{\cite{Ch_metric}}
  \label{strongly-equilateral}
  A graph $G$ is weakly modular if and only if all metric triangles of $G$
  are strongly equilateral.
\end{lemma}

\begin{definition} [Bridged and weakly bridged graphs]
  A graph $G$ is called \emph{bridged} \cite{FaJa,SoCh} if the balls
  $B_r(S)$ around convex sets $S$ of $G$ are convex.  Equivalently, $G$ is
  bridged if all isometric cycles of $G$ have length 3. Bridged graphs are
  weakly modular \cite{Ch_metric}.

  A graph $G=(V,E)$ is called \emph{weakly bridged} \cite{ChOs} if $G$ is a
  weakly modular graph with convex balls. Equivalently, $G$ is weakly
  bridged if $G$ is a weakly modular graph not containing induced 4-cycles
  \cite{ChOs}.  Every bridged graph is weakly bridged.

  A graph $G=(V,E)$ is called \emph{pre-median} if it is weakly modular
  and does not contain an induced $K_{2,3}$ or an induced $W^-_4$
  \cite{Cha1,Cha2}.
\end{definition}

\subsection{Positioning condition}
The positioning condition was introduced by Maurer \cite{Mau} and used in
\cite{Mau} and \cite{Ch_delta} to characterize the basis graphs of matroids
and even $\Delta$-matroids.

The \emph{layering} of a graph $G=(V,E)$ with respect to a vertex $b$ is
the partition ${\mathcal L}(b)=\{ S_k(b), k=0,1,2,\ldots\}$ of $V$ into
concentric spheres centered at $b$. Up to relabeling of its vertices, each
square $v_1v_2v_3v_4$ of $G$ is located in this layering ${\mathcal L}(b)$
in one of the following positions:
\begin{itemize}
\item[(L1)] all vertices in one sphere $S_k(b)$;
\item[(L2)] two adjacent vertices in $S_k(b)$ and two other adjacent
  vertices in the next sphere $S_{k+1}(b)$;
\item[(L3)] in three consecutive spheres $S_{k}(b),S_{k+1}(b)$, and
  $S_{k+2}(b)$ with two non-adjacent vertices in the middle sphere
  $S_{k+1}(b)$;
\item[(L4)] three vertices in $S_k(b)$ and one vertex in the next sphere
  $S_{k+1}(b)$;
\item[(L5)] one vertex in $S_k(b)$ and three vertices in the next sphere
  $S_{k+1}(b)$;
\item[(L6)] two non-adjacent vertices in $S_k(b)$ and two non-adjacent
  vertices in $S_{k+1}(b)$.
\end{itemize}

\begin{definition}[Positioning and Weak Positioning Conditions]
  A graph $G=(V,E)$ satisfies the \emph{positioning condition (PC)} if for
  every vertex $b$ and every square $v_1v_2v_3v_4$ of $G$, in the layering
  ${\mathcal L}(b)$ of $G$, $v_1v_2v_3v_4$ is located in one of the
  positions (L1), (L2), or (L3). Equivalently, $G$ satisfies (PC) if the
  equality $d(b,v_1)+d(b,v_3)=d(b,v_2)+d(b,v_4)$ holds.

  A graph $G=(V,E)$ satisfies the \emph{weak positioning condition (WPC)}
  if for every $b$ and every square $v_1v_2v_3v_4$ of $G$, $v_1v_2v_3v_4$ is located in one of the
  positions (L1), (L2), (L3), or (L4) in the layering
  ${\mathcal L}(b)$ of $G$.
\end{definition}

In conditions (L1)--(L6), $k$ is the distance from $b$ to the square
$v_1v_2v_3v_4$. If $k=1$, then we refer to \emph{local conditions}
(L1)-(L6). If the square $v_1v_2v_3v_4$ satisfies the local condition (L4),
then $v_1,v_2,v_3,v_4,b$ induce the graph $W^-_4$ from Fig.~\ref{f:W4-}. If
$v_1v_2v_3v_4$ satisfies the local condition (L6), then $v_1,v_2,v_3,v_4,b$
induce the graph $K_{2,3}$ from Fig.~\ref{f:forbidden-subgraphs}.  Finally,
if $v_1v_2v_3v_4$ satisfies the local condition (L5), then
$v_1,v_2,v_3,v_4,b$ induce the graph $C_4+e$, such that $v_1v_2v_3v_4$ is
the 4-cycle $C_4$, $bv_1$ is a pendant edge, and $d(b,v_3)=2$.

\begin{figure}[ht]
\centering
\begin{subfigure}[t]{0.33\textwidth}
  \centering
  \resizebox{\linewidth}{!}{
    \begin{tikzpicture}[every node/.style={circle, draw, fill=white, inner sep=1.5pt}]
      \node (u) at (0,0) {$u$};
      \node (v) at (2,2) {$v$};
      \node (w) at (2,-2) {$w$};
      \node (y) at (6,0) {$y$};
      \node (x) at (4,0) {$x$};
      \draw (u) -- (v) -- (y) -- (w) -- (x) -- (v);
      \draw (u) -- (w);
    \end{tikzpicture}}
  \caption{$K_{2,3}$}
\end{subfigure}
\quad
\begin{subfigure}[t]{0.33\textwidth}
  \centering
  \resizebox{\linewidth}{!}{
    \begin{tikzpicture}[every node/.style={circle, draw, fill=white, inner sep=1.5pt}]
      \node (u) at (0,0) {$u$};
      \node (v) at (2,2) {$v$};
      \node (w) at (2,-2) {$w$};
      \node (y) at (6,0) {$y$};
      \node (x) at (4,0) {$x$};
      \draw (u) -- (v) -- (y) -- (w) -- (x) -- (v);
      \draw (u) -- (w) -- (v);
    \end{tikzpicture}}
  \caption{The propeller $K^+_{2,3}$}
\end{subfigure}%
\\
\begin{subfigure}[t]{0.4\textwidth}
  \centering
  \resizebox{\linewidth}{!}{
    \begin{tikzpicture}[every node/.style={circle, draw, fill=white, inner sep=1.5pt}]
      \node (u) at (-6,0) {\footnotesize $u$};
      \node (v) at (-6,2) {\footnotesize $v$};
      \node (x) at (-4,2) {\footnotesize $x$};
      \node (a) at (-2,2) {};
      \node (b) at (-4,0) {};
      \node (w) at (-2,0) {\footnotesize $w$};
      \node (y) at (0,0) {\footnotesize $y$};
      \draw (u) -- (v) --(x) -- (a) -- (y) -- (w) --(b) -- (u);
      \draw (v) -- (b) -- (a) -- (w) --(x) -- (b);
      \draw (v) to[bend left=30] (a);
    \end{tikzpicture}}
  \caption{The graph $R_1$}
\end{subfigure}%
\quad
\begin{subfigure}[t]{0.4\textwidth}
  \centering
  \resizebox{\linewidth}{!}{
    \begin{tikzpicture}[every node/.style={circle, draw, fill=white, inner sep=1.5pt}]
      \node (u) at (-6,0) {\footnotesize $u$};
      \node (v) at (-6,2) {\footnotesize $v$};
      \node (x) at (-4,2) {\footnotesize $x$};
      \node (a) at (-2,2) {};
      \node (b) at (-4,0) {};
      \node (w) at (-2,0) {\footnotesize $w$};
      \node (y) at (0,0) {\footnotesize $y$};
      \draw (u) -- (v) --(x) -- (a) -- (y) -- (w) --(b) -- (u);
      \draw (b) -- (a) -- (w) --(x)--(b);
      \draw (v) to[bend left=30] (a);
    \end{tikzpicture}}
  \caption{The graph $R_2$}
\end{subfigure}%
\caption{Non-smooth graphs.} \label{f:forbidden-subgraphs}
\end{figure}
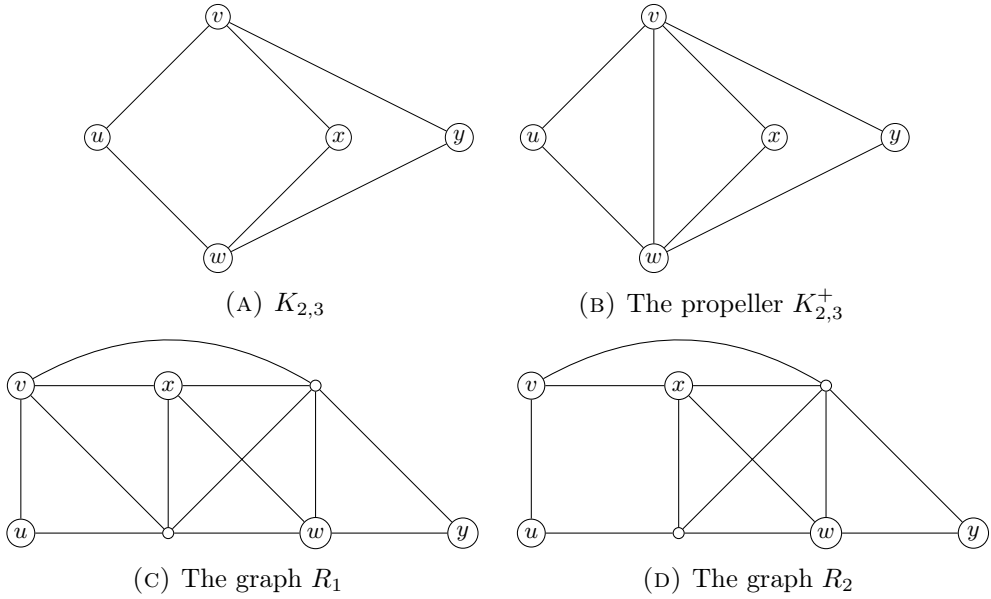

\section{Main results}
In this section, we characterize smooth and strongly smooth weakly modular
graphs in terms of forbidden isometric subgraphs on 5 and 7 vertices and
one forbidden configurations on 5 vertices.  In
Fig.~\ref{f:forbidden-subgraphs} we present four non-smooth graphs (the
first three also occur in \cite{BrChNSScSt}): $K_{2,3}$, the propeller
$K_{2,3}^+$ ($K_{2,3}$ plus an edge between the two vertices of degree 3),
and the graphs $R_1$ and $R_2$.  In all four graphs, the distance
point-shadow $v|u$ is not convex: $x,y\in v|u, w\in I(x,y),$ but
$w\notin v|u$.  The graph $W_4^-$ from Fig.~\ref{f:W4-} is weakly modular,
smooth, but not strongly smooth.  There exists another setting with
non-convex distance point-shadows. It occurs from the local (L5) condition,
discussed above. We say that the vertices $b,v_1,v_2,v_3,v_4$ of a graph
$G$ define the \emph{configuration} $\mathcal{C}$ if they induce in $G$ the
graph $C_4+e$ (the 4-cycle $v_1v_2v_3v_4$ plus the pendant edge $bv_1$)
with the distance constraint $d(b,v_3)=2$. Then, $v_2,v_4\in v_1|b$ and
$v_3\in I(v_2,v_4)\setminus (v_1|b)$, thus $v_1|b$ is not convex.  By
applying the triangle condition, the configuration $\mathcal{C}$ extends to
five non-smooth graphs $D_1$, $D_2$, $D_3$, $D_4$, $D_5$ on 7 vertices,
depicted in Fig.~\ref{f:forbidden-subgraphs-bid}.  We continue with the
formulation of the main results of this paper:

\begin{theorem}
  \label{t:wm-smooth}
  For a weakly modular graph $G=(V,E)$ the following conditions are
  equivalent:
  \begin{itemize}
  \item[(1)] $G$ is smooth;
  \item[(2)] $G$ contains neither the graphs $K_{2,3}$, $K^+_{2,3}$, $R_1$,
    and $R_2$ from Fig.~\ref{f:forbidden-subgraphs} as isometric subgraphs
    nor the configuration $\mathcal{C}$;
  \item[(3)] $G$ does not contains the graphs $K_{2,3},K^+_{2,3}, R_1,$ and
    $R_2$ from Fig.~\ref{f:forbidden-subgraphs} and the graphs
    $D_1,D_2,D_3,D_4,D_5$ from Fig.~\ref{f:forbidden-subgraphs-bid} as
    isometric subgraphs.
  \end{itemize}
\end{theorem}

\begin{theorem}\label{t:wm-strongly-smooth} A weakly modular graph
  $G=(V,E)$ is strongly smooth if and only if $G$ does not contains the
  graphs $W_4^-$, $K_{2,3}$, $K^+_{2,3}$, and $R_1$ from Figures
  \ref{f:W4-} and \ref{f:forbidden-subgraphs} as isometric subgraphs.
\end{theorem}

\begin{theorem}\label{t:wb-smooth-strongly-smooth} For a weakly bridged
  graph $G=(V,E)$, the following conditions are equivalent:
  \begin{itemize}
  \item[(1)] $G$ is strongly smooth;
  \item[(2)] $G$ is smooth;
  \item[(3)] $G$ does not contains $K_{2,3}^+$ and $R_1$ as isometric
    subgraphs.
  \end{itemize}
\end{theorem}

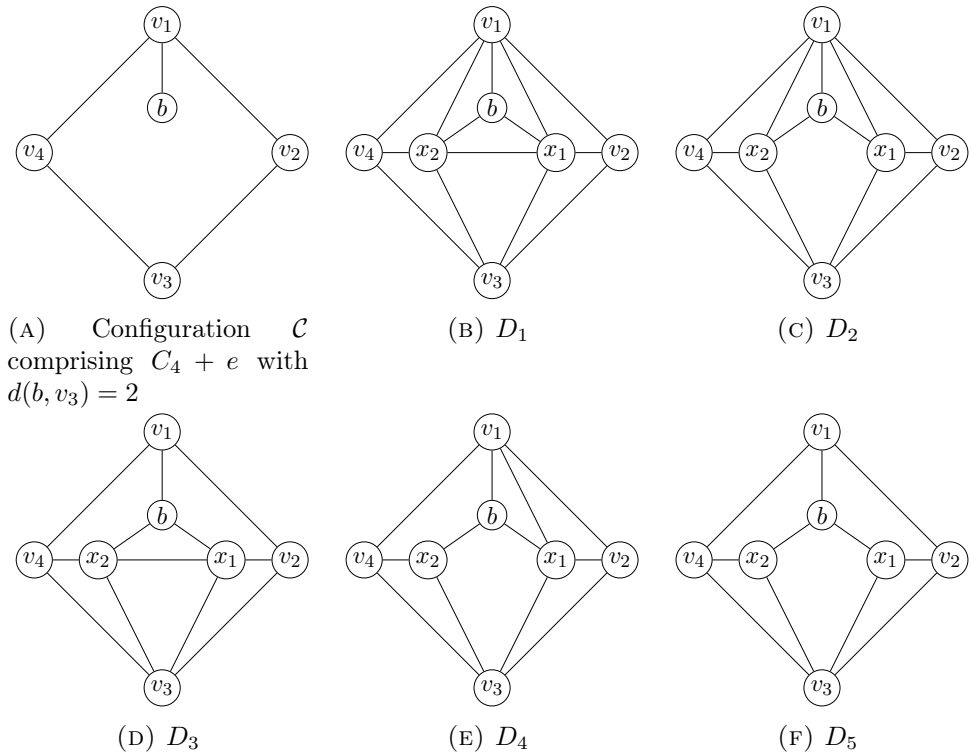
\begin{figure}[ht]
  \centering
    \begin{subfigure}[t]{0.25\textwidth}
      \centering
      \resizebox{\linewidth}{!}{
        \begin{tikzpicture}[every node/.style={circle, draw, fill=white, inner sep=1.5pt}]
          % $K_{2,3}$-draw
          \node (v4) at (-6,0) {$v_4$};
          \node (v1) at (-4,2) {$v_1$};
          \node (v3) at (-4,-2) {$v_3$};
          \node (v2) at (-2,0) {$v_2$};
          \node (b1) at (-4,0.7) {$b$};
          % \node (b2) at (-4,-0.7) {};
          \draw (v1) -- (v2) --(v3) -- (v4) -- (v1) -- (b1);
        \end{tikzpicture}}
      \caption{Configuration $\mathcal{C}$ comprising $C_4+e$ with
        $d(b,v_3)=2$}
    \end{subfigure}%
\quad
\begin{subfigure}[t]{0.25\textwidth}
  \centering
  \resizebox{\linewidth}{!}{
    \begin{tikzpicture}[every node/.style={circle, draw, fill=white, inner sep=1.5pt}]
      % $K_{2,3}$-draw
      \node (v4) at (-6,0) {$v_4$};
      \node (v1) at (-4,2) {$v_1$};
      \node (v3) at (-4,-2) {$v_3$};
      \node (v2) at (-2,0) {$v_2$};
      \node (b1) at (-4,0.7) {$b$};
      % \node (b2) at (-4,-0.7) {$c$};
      \node (x2) at (-5, 0) {$x_2$};
      \node (x1) at (-3, 0) {$x_1$};
      \draw (v1) -- (v2) --(v3) -- (v4) -- (v1) -- (b1);
      \draw (v4) -- (x2) -- (b1) -- (x1) -- (v2);
      \draw (v1) -- (x1) -- (x2) -- (v1);
      \draw (x1) -- (v3);
      \draw (x2) -- (v3);
    \end{tikzpicture}}
  \caption{$D_1$}
\end{subfigure}%
\quad
\begin{subfigure}[t]{0.25\textwidth}
  \centering
  \resizebox{\linewidth}{!}{
    \begin{tikzpicture}[every node/.style={circle, draw, fill=white, inner sep=1.5pt}]
      % $K_{2,3}$-draw
      \node (v4) at (-6,0) {$v_4$};
      \node (v1) at (-4,2) {$v_1$};
      \node (v3) at (-4,-2) {$v_3$};
      \node (v2) at (-2,0) {$v_2$};
      \node (b1) at (-4,0.7) {$b$};
      % \node (b2) at (-4,-0.7) {$c$};
      \node (x2) at (-5, 0) {$x_2$};
      \node (x1) at (-3, 0) {$x_1$};
      \draw (v1) -- (v2) --(v3) -- (v4) -- (v1) -- (b1);
      \draw (v4) -- (x2) -- (b1) -- (x1) -- (v2);
      \draw (v1) -- (x1);
      \draw (v1) -- (x2);
      \draw (x1) -- (v3);
      \draw (x2) -- (v3);
    \end{tikzpicture}}
  \caption{$D_2$}
\end{subfigure}%
\\
\begin{subfigure}[t]{0.25\textwidth}
  \centering
  \resizebox{\linewidth}{!}{
    \begin{tikzpicture}[every node/.style={circle, draw, fill=white, inner sep=1.5pt}]
      % $K_{2,3}$-draw
      \node (v4) at (-6,0) {$v_4$};
      \node (v1) at (-4,2) {$v_1$};
      \node (v3) at (-4,-2) {$v_3$};
      \node (v2) at (-2,0) {$v_2$};
      \node (b1) at (-4,0.7) {$b$};
      % \node (b2) at (-4,-0.7) {};
      \node (x2) at (-5, 0) {$x_2$};
      \node (x1) at (-3, 0) {$x_1$};
      \draw (v1) -- (v2) --(v3) -- (v4) -- (v1) -- (b1);
      \draw (v4) -- (x2) -- (b1) -- (x1) -- (v2);
      \draw (x1) -- (x2);
      \draw (x1) -- (v3);
      \draw (x2) -- (v3);
    \end{tikzpicture}}
  \caption{$D_3$}
\end{subfigure}%
\quad
\begin{subfigure}[t]{0.25\textwidth}
  \centering
  \resizebox{\linewidth}{!}{
    \begin{tikzpicture}[every node/.style={circle, draw, fill=white, inner sep=1.5pt}]
      % $K_{2,3}$-draw
      \node (v4) at (-6,0) {$v_4$};
      \node (v1) at (-4,2) {$v_1$};
      \node (v3) at (-4,-2) {$v_3$};
      \node (v2) at (-2,0) {$v_2$};
      \node (b1) at (-4,0.7) {$b$};
      % \node (b2) at (-4,-0.7) {};
      \node (x2) at (-5, 0) {$x_2$};
      \node (x1) at (-3, 0) {$x_1$};
      \draw (v1) -- (v2) --(v3) -- (v4) -- (v1) -- (b1);
      \draw (v4) -- (x2) -- (b1) -- (x1) -- (v2);
      \draw (v1) -- (x1);
      \draw (x1) -- (v3);
      \draw (x2) -- (v3);
    \end{tikzpicture}}
  \caption{$D_4$}
\end{subfigure}%
\quad
\begin{subfigure}[t]{0.25\textwidth}
  \centering
  \resizebox{\linewidth}{!}{
    \begin{tikzpicture}[every node/.style={circle, draw, fill=white, inner sep=1.5pt}]
      % $K_{2,3}$-draw
      \node (v4) at (-6,0) {$v_4$};
      \node (v1) at (-4,2) {$v_1$};
      \node (v3) at (-4,-2) {$v_3$};
      \node (v2) at (-2,0) {$v_2$};
      \node (b1) at (-4,0.7) {$b$};
      % \node (b2) at (-4,-0.7) {};
      \node (x2) at (-5, 0) {$x_2$};
      \node (x1) at (-3, 0) {$x_1$};
      \draw (v1) -- (v2) --(v3) -- (v4) -- (v1) -- (b1);
      \draw (v4) -- (x2) -- (b1) -- (x1) -- (v2);
      \draw (x1) -- (v3);
      \draw (x2) -- (v3);
    \end{tikzpicture}}
  \caption{$D_5$}
\end{subfigure}%

\caption{Configuration $\mathcal{C}$ and its extensions $D_1$, $D_2$,
  $D_3$, $D_4$, $D_4$, $D_5$.}
\label{f:forbidden-subgraphs-bid}
\end{figure}

\subsection{Proof of Theorem \ref{t:wm-smooth}}

(1)$\Rightarrow$(2) If a weakly modular graph $G$ contains the
configuration $\mathcal{C}$, then $d(b,v_3)=2$ implies that the distance
point-shadow $v_1|b$ is not convex: $v_2,v_4\in v_1|b$ and
$v_3\in I(v_2,v_4)\setminus (v_1|b)$. If $G$ contains one of the graphs
$K_{2,3}$, $K^+_{2,3}$, $R_1$, and $R_2$ from
Fig.~\ref{f:forbidden-subgraphs} as an isometric subgraph, then the
distance point-shadow $v|u$ is not convex: $x,y\in v|u$ and
$w\in I(x,y)\setminus (v|u)$. Consequently, if $G$ is smooth, then $G$ does
not contain an isometric $K_{2,3}$, $K^+_{2,3}$, $R_1,$, $R_2$, or
configuration $\mathcal{C}$.

(2)$\Rightarrow$(3) The configuration $\mathcal{C}$ is present in each of
the graphs $D_1,\ldots,D_5$. Hence if $G$ does not contain $\mathcal{C}$,
then $G$ does not contain any of $D_1,\ldots,D_5$ as an induced
subgraph. Since the diameter of each of these graphs is $2$, $G$ does not
contain $D_1,\ldots,D_5$ as isometric subgraphs.

(3)$\Rightarrow$(2) Suppose that a weakly modular graph $G$ does not
contain the graphs from Fig.~\ref{f:forbidden-subgraphs} as isometric
subgraphs, but $G$ contains the configuration $\mathcal{C}$. We will show
that $G$ contains one of the graphs $D_1,\ldots,D_5$ as an induced (and
thus isometric) subgraph. Applying (TC) twice to $b$ and the edges $v_2v_3$
and $v_3v_4$ of $\mathcal{C}$, we will find two vertices
$x_1\sim b,v_2,v_3$ and $x_2\sim b,v_3,v_4$. If $x_1=x_2$, then the
vertices $b,v_2,v_4,x,v_1$ induce either the forbidden $K_{2,3}$ or the
forbidden $K^+_{2,3}$. Now, suppose that $x_1\ne x_2$. Then $x_1\nsim v_4$
and $x_2\nsim v_2$. Suppose not, and let say $x_1\sim v_4$ (the case
$x_2\sim v_2$ is similar). Then $v_1,x_1,v_2,b,v_4$ induce a forbidden
$K_{2,3}$ or $K_{2,3}^+$. Consequently, $x_1\nsim v_4, x_2\nsim v_2$.  Then
up to isomorphism we get the 5 graphs $D_1,\ldots,D_5$ depending which of
the vertices $x_1,x_2,v_1$ are adjacent. If $x_1,x_2,v_1$ are pairwise
non-adjacent, we get $D_5$ and if $v_1$ is adjacent to only one of the
vertices $x_1,x_2$, we get $D_4$. If only $x_1$ and $x_2$ are adjacent, we
get $D_3$, if $v_1$ is adjacent to both $x_1,x_2$, we get $D_2$, and if
$x_1,x_2,v_1$ are pairwise adjacent, we get $D_1$. Consequently, if a
weakly modular graph $G$ does not contains the graphs
$K_{2,3},K^+_{2,3}, R_1,$ and $R_2$ from Fig.~\ref{f:forbidden-subgraphs}
and the graphs $D_1,D_2,D_3,D_4,D_5$ from
Fig.~\ref{f:forbidden-subgraphs-bid} as isometric subgraphs, then $G$ does
not contain the configuration $\mathcal{C}$.

(2)$\Rightarrow$(1) Suppose that $G$ is a weakly modular graph that does
not contain the configuration $\mathcal{C}$ and has none of the graphs
$K_{2,3}$, $K^+_{2,3}$, $R_1$, and $R_2$ as isometric subgraphs. First we
prove the following property of $G$.

\begin{claim-swm} [Weak positioning condition]
  \label{wposcond}
  % If $G$ satisfies the conditions of Theorem \ref{t:wm-smooth}, then
  $G$ satisfies the weak positioning condition (WPC).
\end{claim-swm}
\begin{proof}
  Let $b$ be any vertex and $v_1v_2v_3v_4$ be any square of $G$. Suppose by
  way of contradiction that $b$ and $v_1v_2v_3v_4$ violate (WPC). Then
  $v_1v_2v_3v_4$ is located in one of the positions (L5) or (L6).  First,
  suppose that $v_1,v_3\in S_{k+1}(b)$ and $v_2,v_4\in S_k(b)$ (position
  (L6)). Then $v_2,v_4\in I(b,v_1)\cap I(b,v_3)$.  By (QC) there exists
  $u\sim v_2,v_4$ at distance $k-1$ from $b$. Then $u\nsim v_1,v_3$ and
  $u,v_1,v_2,v_3,v_4$ induce the forbidden $K_{2,3}$.

  Now suppose that $d(b,v_1)=k$ and $d(b,v_2)=d(b,v_3)=d(b,v_4)=k+1$
  (position (L5)). Since $v_1\nsim v_3$, necessarily $b\ne v_1$, i.e.,
  $k>0$. Since $d(b,v_2)=d(b,v_3)=k+1$, by (TC) there exists
  $u\sim v_2,v_3$ at distance $k$ from $b$ and different from $v_1$.  Since
  $u,v_1\in I(b,v_2)$ by (QC) (or by (TC) if $u\sim v_1$) there exists a
  vertex $z\sim u,v_1$ at distance $k-1$ from $b$. This implies that
  $d(z,v_2)=d(z,v_3)=d(z,v_4)=2=d(v_1,v_3)$.  Since $v_1\sim z,v_2,v_4$,
  the vertices $z,v_1,v_2,v_3,v_4$ form the configuration
  $\mathcal{C}$. Since the conditions (2) and (3) are equivalent, this
  concludes the proof that $G$ satisfies (WPC).
\end{proof}

\begin{claim-swm}[Smoothness]
  \label{smoothness}
  $G$ is smooth.
\end{claim-swm}
\begin{proof} From the definition of smooth graphs, it suffice to prove
  that any distance point-shadow $v|u$ with $u\sim v$ is convex.  Since
  $v|u$ is connected (Lemma \ref{shadow-connected}), to prove that $v|u$ is
  convex by Lemma \ref{thm:local-convex->global-convex} it suffices to
  prove that $v|u$ is locally convex. Consider two vertices $x,y\in v|u$ at
  distance $2$ in $v|u$, let $z$ be a common neighbor of $x$ and $y$
  belonging to $v|u$ and $w$ be any other common neighbor of $x$ and $y$ in
  $G$. We prove that $w$ belongs to $v|u$ by induction on the distance sum
  $\sigma(u,v,x,y,z,w)=d(v,x)+d(v,y)+d(v,z)+d(v,w)$.

  Suppose by way of contradiction that $w$ does not belong to $v|u$ and
  that $u,v,x,y,z,w$ is a sextuple minimizing $\sigma(u,v,x,y,z,w)$ among
  the sextuples violating the local convexity of $v|u$.  Since
  $w\notin v|u$, we have $d(u,w)\le d(v,w)$.  Suppose without loss of
  generality that $k\coloneqq d(v,x)\le d(v,y)$. Since $d(x,y)=2,$ we
  obtain $d(v,y)\le k+2$. If $d(v,y)=k+2$, then $x\in I(y,v)$ and
  $w\in I(x,y)$ imply $w\in I(y,v)\subseteq v|u$. Therefore, we may assume
  that $d(v,y)\in \{ k,k+1\}$.

  \begin{case-kk}
    $d(v,x)=d(v,y)=k$.
  \end{case-kk}

  In the basis case $k=1$ we obtain either a forbidden $K_{2,3}$,
  $K_{2,3}^+$ (if $u\sim w$), or the configuration $\mathcal{C}$ (if
  $d(u,w)=2$). Now, assume $k>1$. Since $x,y\in v|u$, we have
  $d(u,x)=d(u,y)=k+1$. Since $w\notin v|u$, $d(u,w)\le d(v,w)$.  Notice
  that we can suppose that $d(v,z)\le k$. Indeed, if $d(v,z)=k+1$, then
  $x,y\in I(z,v)$, thus by (QC) there exists $z'\sim x,y$ at distance $k-1$
  from $v$. Then $z'\in I(x,v)\subseteq v|u$, thus we can replace $z$ by
  $z'$ and obtain a violating sextuple $u,v,x,y,z',w$ with
  $\sigma(u,v,x,y,z',w)<\sigma(u,v,x,y,z,w)$.  We distinguish three
  subcases.

  \begin{subcase-kkk}
    $d(v,w)=k$ and $z\sim w$.
  \end{subcase-kkk}

  Then $d(u,w)\le k$. Since $d(u,x)=k+1$ and $x\sim w$, $d(u,w)=k$.  If
  $d(v,z)=k-1$, then $d(u,z)=k$. By (TC) there exists $t\sim z,w$ at
  distance $k-1$ from $u$. Since $d(u,x)=d(u,y)=k+1$, $t$ cannot be
  adjacent to $x$ and $y$. Consequently, $t,x,y,z,w$ induce a forbidden
  propeller $K^+_{2,3}$.

  So, suppose  $d(v,z)\ge k$, i.e., $d(v,z)=k$. By (TC), there
  exists a vertex $s\sim z,w$ at distance $k-1$ to $v$. Since $z\in v|u,$
  we get $s\in I(v,z)\subseteq v|u$. If $s$ is not adjacent to $x$ or
  $y$, say $s\nsim y$, then $z,w\in I(s,y),$ thus replacing $x$ by $s$ we
  get a sextuple $u,v,s,y,z,w$ violating the local convexity of $v|u$ and
  for which $\sigma(u,v,s,y,z,w)<\sigma(u,v,x,y,z,w)$, contrary to the
  minimality choice of $u,v,x,y,z,w$. So, $s\sim x,y$. In this case
  $s\in I(x,y)\cap (v|u)$. Replacing $z$ by $s$, we get a sextuple
  $u,v,x,y,s,w$, violating local convexity of $v|u$ for which
  $\sigma(u,v,x,y,s,w)<\sigma(u,v,x,y,z,w)$, contrary to our minimality
  choice. This concludes the analysis of this subcase.

  \begin{subcase-kkk}
    $d(v,w)=k$ and $z\nsim w$.
  \end{subcase-kkk}

  In this case, $xzyw$ is a square of $G$. By Claim \ref{wposcond}, $G$
  satisfies the weak positioning condition (WPC), which we apply to the
  square $xzyw$ and the basis points $v$ and $u$.  First, since
  $d(v,x)=d(v,y)=d(v,w)=k$, by (WPC) applied to $v$, we conclude that
  $d(v,z)\ne k-1$.  Therefore $d(v,z)=k$ or $d(v,z)=k+1$. Since
  $d(v,z)=k+1$ is impossible, necessarily we have $d(v,z)=k$. Since
  $z\in v|u$, we get $d(u,z)=k+1$. Since $w\notin v|u$, we get
  $d(u,w)\le d(v,w)=k$. Since $w\sim x$ and $d(u,x)=k+1,$ $d(u,w)=k-1$ is
  impossible. Therefore, $d(u,w)=k$. Since $d(u,x)=d(u,y)=d(u,z)=k+1$, this
  contradicts (WPC) applied to $xzyw$ and $u$. %This contradiction
  %finishes the analysis of this subcase.

  \begin{subcase-kkk}
    $d(v,w)=k+1$.
  \end{subcase-kkk}

  Then $x,y\in I(v,w)$ and by (QC) there exists $z'\sim x,y$ at distance
  $k-1$ from $v$. Then $z'\in I(x,v)\subseteq v|y$, thus we can suppose
  without loss of generality that $z=z'$. Consequently, $d(u,x)=d(u,y)=k+1$
  and $d(u,z)=k$. Applying condition (WPC) to the square $xzyw$ with
  respect to $u$, we conclude that $d(u,w)=k+2$ (otherwise, $xzyw$ will be
  located as in (L5) or (L6)). But then $w\in v|u$, contradicting our
  assumption.  This finishes the analysis of this subcase and of Case 1.

  \begin{case-kk}
    $d(v,x)=k$ and $d(v,y)=k+1$.
  \end{case-kk}

  Since $w$ is adjacent to $x$ and $y$, we have $k\le d(v,w)\le k+1$. If
  $d(v,w)=k$, then $w\in I(v,y)\subseteq v|u$, contrary to our
  assumption. Thus $d(v,w)=k+1$. By (TC) there exists a vertex $z'\sim w,y$
  at distance $k$ to $v$. Then $z'\in I(y,v)\subseteq v|u$. If $x\nsim z'$,
  then $x,z'\in I(w,v)$ and by (QC) there exists $t\sim x,z'$ at distance
  $k-1$ from $v$.  Then $t\in I(x,v)\subseteq v|u$ and we can consider the
  sextuple $u,v,x,z',t,w$. This sextuple violates the local convexity of
  $v|u$ and $\sigma(u,v,x,z',t,w)<\sigma(u,v,x,y,z,w)$, contrary to our
  minimality choice. Therefore $x\sim z'$ and since $z'\in I(x,y)$, we can
  consider $z'$ as the vertex $z$. Consequently, $z\sim w$ and
  $d(v,z)=k$. Since $d(v,x)=d(v,z)=k$, by (TC) there exists $t\sim x,z$ at
  distance $k-1$ from $v$.

  Since $w\notin v|u$, we have $d(u,w)\le d(v,w)=k+1$. From $w\sim x,y$,
  $d(u,x)=k+1$, and $d(u,y)=k+2$, we conclude that $d(u,w)=k+1$. By (TC)
  there exists a vertex $r\sim z,w$ at distance $k$ to $u$. To avoid a
  propeller $K_{2,3}^+$ induced by $r,x,y,z,w$, the vertices $r$ and $x$
  must be adjacent. We distinguish two subcases.

  \begin{subcase-kkk+1}
    $t\sim r$.
  \end{subcase-kkk+1}

  Since $d(u,t)=d(u,r)=k$, by (TC) there exists a vertex $p\sim t,r$ at
  distance $k-1$ from $u$. From the values of the distances from $u$ to the
  vertices $x,z,w,y$ we conclude that $p\nsim x,z,w$ and
  $d(p,y)=3$. Consequently, $p,t,r,x,z,w,y$ induce $R_1$ as an isometric
  subgraph.
  
  \begin{subcase-kkk+1}
    $t\nsim r$.
  \end{subcase-kkk+1}

  In this case we have $t,r\in I(u,z)$ (also $t,r\in I(u,x)$), thus by (QC)
  there exists a vertex $p\sim t,r$ at distance $k-1$ to $u$. Analogously
  to the previous subcase, from the values of distances from $u$ to the
  vertices $x,z,w,y$ we conclude that $p\nsim x,z,w$ and
  $d(p,y)=3$. Consequently, the vertices $p,t,r,x,z,w,y$ induce the
  forbidden graph $R_2$ as an isometric subgraph. This contradiction
  concludes the analysis of this subcase and of Case 2.

  In summary, the assumption that there exists a sextet $u,v,x,y,z,w$ such
  that $x,y\in v|u, d(x,y)=2$, and $z,w\in I(x,y)$ with $z\in v|u$ and
  $w\notin v|u$ always leads us to a contradiction, i.e., such a sextet
  does not exist. Thus the distance point-shadow $v|u$ is locally-convex,
  and thus is convex. Consequently, any weakly modular graph $G$ not
  containing the configuration $\mathcal{C}$ and an isometric $K_{2,3}$,
  $K^+_{2,3}$, $W^-_4$, $R_1$, and $R_2$ is smooth.
\end{proof}

\subsection{Proof of Theorem \ref{t:wm-strongly-smooth}}

We noted above that $W^-_4$ is not strongly smooth and that the graphs
$K_{2,3}$, $K^+_{2,3}$, and $R_1$ are not smooth. Since each strongly
smooth graph is smooth by Lemma~\ref{strongly-smooth->smooth}, $K_{2,3}$,
$K^+_{2,3}$, and $R_1$ are not strongly smooth either. If one of these four
graphs occurs as an isometric subgraph of a weakly modular graph $G$, then
the graph $G$ cannot be strongly smooth. Indeed, if $G$ contains $W^-_4$,
$K_{2,3}$, or $R_1$, $G$ will contain a non-convex interval and
therefore cannot be strongly smooth by Lemma
\ref{strongly-smooth->smooth}. If $G$ contains the propeller $K^+_{2,3}$,
then the shadow $v/u$ is not convex since $x,y\in v/u$ and $w\in I(x,y),$
however $\conv(u,w)=\{ u,w\}$ does not contain $v$.

Conversely, suppose that $G$ is a weakly modular graph not containing
$W^-_4$, $K_{2,3}$, $K^+_{2,3}$, and $R_1$ as isometric subgraphs. First we
prove two properties of such graphs. Our proofs closely follow the proofs
of more general formulations for meshed graphs from \cite{ChS3}.

\begin{claim-sswm}[Intervals are convex]
  \label{convex-intervals}
  $G$ is a graph with convex intervals.
\end{claim-sswm}

\begin{proof}
  Suppose by way of contradiction that $G$ has non-convex intervals. Let
  $I(u,v)$ be a non-convex interval minimizing $k=d(u,v)$.  Since $I(u,v)$
  is connected, by Lemma \ref{thm:local-convex->global-convex} it follows
  that $I(u,v)$ is not locally convex. Then there exist two vertices
  $x,y\in I(u,v)$ at distance 2 having a common neighbor $z\notin
  I(u,v)$. From the minimality choice of $k=d(u,v)$, it follows that
  $I(u,x)\cap I(u,y)=\{ u\}$ and $I(v,x)\cap I(v,y)=\{v\}$. Therefore, all
  quasi-medians of the triplets $u,x,y$ and $v,x,y$ have the form $ux'y'$
  and $vx''y''$, respectively.  Since $d(x',y'),d(x'',y'')\le d(x,y)\le 2$
  and metric triangles $ux'y'$ and $vx''y''$ are strongly equilateral
  (Lemma \ref{strongly-equilateral}), we conclude that
  $\max\{ d(u,x),d(u,y),d(v,x),d(v,y)\}\le 2$. Thus $d(u,v)\le 4$.

  \begin{case-int}
    $d(u,v)=4$.
  \end{case-int}

  Then $uxy$ and $vxy$ are metric triangles with side length
  2. Since these metric triangles are strongly equilateral and
  $z\in I(x,y)$, we conclude that $d(u,z)=2=d(v,z).$ Consequently,
  $z\in I(u,v)$, contrary to the choice of $z$. So, $d(u,v)\le 3$. If
  $d(u,v)=2$, then $u,v,x,y,z$ induce one of the forbidden graphs
  $K_{2,3}$ or $W^-_4$ depending of whether or not $z$ is adjacent to one
  of $u$ or $v$ (if $z\sim u,v$, then $z\in I(u,v)$, a
    contradiction). Thus $d(u,v)=3$.

  \begin{case-int}
    $u$ is adjacent to $x$ and $y$.
  \end{case-int}

  Then $x,y\in I(u,v)$ and by (QC) there exists $w\sim x,y,v$. Then
  $w\in I(v,x)\cap I(v,y)$, contrary to the assumption that
  $I(v,x)\cap I(v,y)=\{ v\}$.  Analogously, $v$ cannot be adjacent to both
  $x$ and $y$. So, further we assume that $d(u,v)=3$ and $x$ and $y$ are
  not both adjacent to $u$ or to $v$.
  
  \begin{case-int}
    $u\sim x$ and $v\sim y$.
  \end{case-int}

  In this case, any quasi-median of the triplet $u,x,y$ is a metric
  triangle with side length 1 of the form $uxy'$, where $y'\sim
  y$. Analogously, any quasi-median of the triplet $v,x,y$ is a metric
  triangle $vx''y$ with side length 1 with $x''\sim x$. Since $d(u,v)=3$,
  $y'\ne x''$. Since $y',x''\in I(u,v),$ $z\ne y',x''$. To avoid a
  $K_{2,3}$ or $W^-_4$ induced by $x,y,y',x'',z$, at least two of the pairs
  $\{y',z\}, \{y',x''\}, \{z,x''\}$ must be edges of $G$. If $y',x'',z$ are
  pairwise adjacent, then $u,v,x,y,x'',y',z$ induce the forbidden $R_1$ as
  an isometric subgraph. If $y'\sim z,x''$ and $z\nsim x''$, then
  $u,x,y',x'',z$ induce the forbidden propeller $K^+_{2,3}$ (if
  $x''\sim y',z$ and $y'\nsim z$, then $K^+_{2,3}$ is induced by
  $y',z,v,x',y$). Finally suppose that $z\sim y',x''$ and $y'\nsim x''$. By
  (TC) there exists $t\sim x,y',v$. To avoid a $K^+_{2,3}$ induced by
  $u,x,y',z,t$, $t$ must be adjacent to $u$ or to $z$. Since $d(u,v)=3,$
  $t$ cannot be adjacent to $u$. Hence $t\sim z$. If $t\nsim y$, then
  $y',z,t,y,v$ induce the forbidden $W^-_4$. Thus $t\sim y$. Then the
  vertices $u,x,y',z,t,v$ induce a forbidden isometric $R_1$. This
  concludes the proof of convexity of the intervals of $G$.
\end{proof}

\begin{claim-sswm}[Positioning condition]
  \label{poscond}
  $G$ satisfies the positioning condition (PC).
\end{claim-sswm}

\begin{proof} Let $b$ be any vertex and $S=v_1v_2v_3v_4$ be any square of
  $G$. We proceed by induction on the distance sum
  $\sigma(v_1,v_2,v_3,v_4,b)=d(b,v_1)+d(b,v_2)+d(b,v_3)+d(b,v_4)$. Suppose
  by way of contradiction that $d(b,v_1)+d(b,v_3)\ne
  d(b,v_2)+d(b,v_4)$. Then $S$ is located in one of the three positions:
  (L4), (L5), or (L6). We consider the following cases.

  \begin{case-pos}
    $d(b,v_1)=d(v,v_3)=k+1$ and $d(b,v_2)=d(b,v_4)=k$ (position (L6)).
  \end{case-pos}

  By (QC) there exists $s\sim v_2,v_4$ at distance $k-1$ from $b$. Then
  $s\nsim v_1,v_3$ and $s,v_1,v_2,v_3,v_4$ induce the forbidden $K_{2,3}$.

  \begin{case-pos}
    $d(b,v_1)=d(b,v_2)=d(b,v_4)=k$ and $d(b,v_3)=k+1$
    (position (L4)).
  \end{case-pos}

  By (QC), there exists a vertex $s\sim v_2,v_4$ at distance $k-1$ from $b$
  (and thus not adjacent to $v_3$). If $s\nsim v_1$, then
  $v_1,v_2,v_3,v_4,s$ induce a forbidden $K_{2,3}$ and if $s\sim v_1$, then
  the same vertices induce the forbidden $W^-_4$.

  \begin{case-pos}  $d(b,v_1)=k$ and $d(b,v_2)=d(b,v_3)=d(b,v_4)=k+1$
    (position (L5)).
  \end{case-pos}
  Since $v_1\nsim v_3$, we have $b\ne v_1$ and $k>0$.  By (TC) there exists
  $s\sim v_2,v_3$ at distance $k$ from $b$. If $s\sim v_1$, then to avoid
  that $s,v_1,v_2,v_3,v_4$ induce a $W^-_4$, we must have $s\sim v_4$.
  Analogously, if $s\sim v_4$ and $s\nsim v_1$, then the vertices
  $s,v_1,v_2,v_3,v_4$ induce the a forbidden $W^-_4$. Therefore $s$ is
  either adjacent to both $v_1$ and $v_4$ or to neither of these vertices.
  First suppose that $s\sim v_1,v_4$. Since $d(b,s)=d(b,v_1)=k$, by (TC)
  there exists $t\sim s,v_1$ at distance $k-1$ from $b$. Then
  $t,s,v_1,v_2,v_4$ induce the forbidden propeller $K^+_{2,3}$.

  Now, let $s\nsim v_1,v_4$. Since $s,v_1\in I(b,v_2)$, by (QC) there
  exists $z\sim v_1,s$ at distance $k-1$ from $b$. Then $z\nsim v_4$ and
  the square $v_2szv_1$ and the vertex $v_4$ violate (PC) because
  $d(v_4,v_1)+d(v_4,s)=3$ and $d(v_4,v_2)+d(v_4,z)=4$. If $b\ne z$, then
  $\sigma(v_2,s,z,v_1,v_4)<\sigma(v_1,v_2,v_3,v_4,b)$ and we obtain a
  contradiction with the minimality choice of $b$ and
  $v_1,v_2v_3v_4$. Consequently, $b=z$, $k=1$, and
  $\sigma(v_1,v_2,v_3,v_4,b)=7$. Since $d(s,v_1)=d(s,v_4)=2$, by (TC) there
  exists $w\sim s,v_1,v_4$. Since $v_2\nsim b$, to avoid a forbidden
  $K_{2,3}$ or $W^-_4$ induced by $z,w,v_2,s,v_1$, $w$ must be adjacent to
  $z$ and $v_2$. But then the vertices $b,v_2,v_4,w,v_1$ induce the
  forbidden $K^+_{2,3}$.  This concludes the analysis of the last case and
  of the proof that $G$ satisfies (PC).
\end{proof}

We are ready to complete the proof of Theorem \ref{t:wm-strongly-smooth}.

\begin{claim-sswm}[Strong smoothness]
  \label{strongly-smooth}
  $G$ is strongly smooth.
\end{claim-sswm}

\begin{proof}
  By Proposition \ref{convex-intervals}, $G$ has convex intervals, thus
  $v/u=v|u$ for any $u,v\in V$. Therefore, it suffices to show that
  all distance point-shadows $v|u$ of $G$ are convex, i.e., that $G$ is
  smooth. Since $G$ does not contain an induced $K_{2,3}$, $K^+_{2,3}$, or
  $R_1$, by Theorem \ref{t:wm-smooth} it suffices to show that $G$ does not
  contain the configuration $\mathcal{C}$ or an isometric $R_2$. Since
  $R_2$ contains a forbidden $W^-_4$, we conclude that $G$ also does not
  contain an isometric $R_2$.
  
  Suppose that $G$ contains the configuration $\mathcal{C}$, i.e., five
  vertices $u,v,x,y,w$ such that
  $u\sim v, v\sim x,y, u\nsim x,y, x\nsim y, w\sim x,y$ and
  $d(u,w)=d(v,w)=2$. Since $d(u,x)=d(u,w)=2$ and $d(u,y)=d(u,w)=2$, by (TC)
  there exist vertices $s$ and $t$ such that $s\sim u,x,w$ and
  $t\sim u,y,w$.  If $s=t$, then the vertices $u,v,s,x,y$ induce the
  forbidden $K_{2,3}$ if $s\nsim v$, and the forbidden $K_{2,3}^+$ if
  $s\sim v$. Therefore, $s\ne t$. This implies $s\nsim y$ and $t\nsim
  x$. If the vertices $s$ and $t$ are not adjacent, then necessarily $v$ is
  adjacent at least to one of the vertices $s,t$, otherwise $v$ and the
  square $uswt$ violate the positioning condition (PC), contrary to
  Proposition \ref{poscond}. If say $s\sim v$, then the vertices
  $s,x,v,w,y$ induce the forbidden $W^-_4$. So suppose that $s\sim t$. If
  $s\nsim v$, then the vertex $y$ and the square $usxv$ violate (PC)
  because $d(y,s)+d(y,v)=3$ and $d(y,u)+d(y,x)=4$. Hence $v\sim
  s$. Consequently, the vertices $s,v,x,y,w$ induce the forbidden
  $W^-_4$. This shows that $G$ cannot contain the configuration
  $\mathcal{C}$. Hence $G$ is smooth and therefore strongly smooth.
\end{proof}

\subsection{Proof of Theorem \ref{t:wb-smooth-strongly-smooth}}
The implication (1)$\Rightarrow$(2) follows from Lemma
\ref{strongly-smooth->smooth}. If $G$ is smooth, then by Theorem
\ref{t:wm-smooth} $G$ does not contain $K_{2,3}^+$ and $R_1$ as isometric
subgraphs, thus (2)$\Rightarrow$(3). Finally, suppose that a weakly bridged
graph $G$ does not contain $K_{2,3}^+$ and $R_1$ as isometric
subgraphs. Since $G$ does not contain induced 4-cycles, $G$ does not
contain $K_{2,3}$, $W^-_4$, and $R_2$ as induced (and thus as isometric)
subgraphs. Consequently, $G$ is a weakly modular graph not containing
$K_{2,3}$, $K^+_{2,3}$, $W^-_4$, $R_1,$ and $R_2$ as isometric
subgraphs. By Theorem \ref{t:wm-strongly-smooth}, $G$ is strongly smooth,
yielding (3)$\Rightarrow$(1). This concludes the proof of Theorem
\ref{t:wb-smooth-strongly-smooth}.

\section{On prime strongly smooth weakly modular graphs} 

Bre{\v{s}}ar et al.\ \cite{BrChNSScSt} proved that smooth graphs are closed
under Cartesian products and gated amalgamations (this also holds for
strongly smooth graphs) and asked what the prime smooth weakly modular
graphs are.  In this section, we give an answer to this question for the
case of prime strongly smooth weakly modular graphs.

Before presenting the result, we recall some facts about primes for other
classes of weakly modular graphs. A classical fact about median graphs is a
result of Isbell \cite{Is} that any finite median graph can be obtained
from hypercubes (Cartesian products of edges) by a sequence of gated
amalgamations. Hence $K_2$ (an edge) is the only prime median
graph. Analogously, the complete graphs $K_n$, $n\ge 2$ are the only prime
quasi-median graphs \cite{BaMuWi}.  A similar result was obtained in
\cite{BaCh_weak} for weakly median graphs (weakly modular graphs in which
each triplet of vertices has a unique quasi-median): the prime weakly
median graphs are the 5-wheel $W_5$, the octahedra $O_d$, and their
2-connected subgraphs, and 2-connected $K_4$-free plane bridged
triangulations. Another class for which the primes are known is the class
of bucolic graphs (a common generalization of median and weakly bridged
graphs) \cite{BrChChGoOs}: the prime bucolic graphs are the 2-connected
weakly bridged graphs.  Generalizing the techniques of
\cite{BaCh_weak,BaMuWi}, Chastand \cite{Cha1,Cha2} introduced pre-median
graphs (which generalizes median, quasi-median, weakly median, and bucolic
graphs) and showed that for pre-median graphs this kind of decomposition
theorems into primes hold.

Chalopin et al.\ \cite{CCHO} provided the following characterization of
prime pre-median graphs:
\begin{theorem} \upshape{\cite{CCHO}}
  \label{prime-pre-median}
  For a 2-connected pre-median graph $G$, the following conditions are
  equivalent:
  \begin{itemize}
  \item[(i)] $G$ is prime;
  \item[(ii)] each square $C_4$ of $G$ is included in a 4-wheel $W_4$ or in
    the graph $M_4$ (see \cite[Fig.4.1]{CCHO});
  \item[(iii)] the clique complex of $G$ is simply connected\footnote{The
      clique complex of $G$ is \emph{simply connected} if each cycle of $G$
      is null-homotopic with respect to the triangles.}
  \end{itemize}
\end{theorem}
This result shows that the prime pre-median graphs are quite general.

Smooth weakly modular graphs are in general not pre-median, and hence the
results of \cite{Cha1,Cha2} and \cite{CCHO} do not apply to them. Strongly
smooth weakly median graphs, on the other hand, are pre-median
graphs. Therefore, combining Theorem \ref{prime-pre-median} of \cite{CCHO}
and our Theorem \ref{t:wm-strongly-smooth} we obtain the following
characterization of such graphs:
\begin{corollary}
  \label{prime-ssmooth}
  For a 2-connected weakly modular graph $G$ the following conditions are
  equivalent:
  \begin{itemize}
  \item[(i)] $G$ is a prime strongly smooth graph;
  \item[(ii)] each square $C_4$ of $G$ is included in $W_4$ or in $M_4$ and
    $G$ does not contain $K_{2,3},W^-_4,K^+_{2,3}$, and $R_1$ as isometric
    subgraphs;
  \item[(iii)] the clique complex of $G$ is simply connected and $G$ does
    not contain $K_{2,3},W^-_4,K^+_{2,3}$, and $R_1$ as isometric
    subgraphs.
  \end{itemize}
\end{corollary}

We conclude with several examples of prime strongly smooth weakly modular
graphs. First, the prime weakly median graphs are prime strongly
smooth. Recall that they are: (1) \emph{the 5-wheel $W_5$, the octahedra
  $O_d$ and their 2-connected subgraphs (in particular, cliques), and the
  2-connected $K_4$-free plane bridged triangulations}. It was shown in
\cite[Proposition 4.16]{CCHO} that the following graphs are prime
pre-median: half-cubes, Johnson graphs, the Schl\"{a}fli graph, and the
Gosset graph (for definitions, see \cite{CCHO}). Half-cubes and Johnson
graphs do not contain propellers $K^+_{2,3}$ since they are basis graphs of
even $\Delta$-matroids and matroids \cite{Ch_delta,Mau} and their intervals
are convex \cite{Ch_delta}. Therefore, half-cubes and Johnson graphs do not
contain $R_1$ and $R_2$. Consequently, (2) \emph{half-cubes and Jonhson
  graphs} are strongly smooth and prime. One can also check that (3)
\emph{the Schl\"{a}fli graph and the Gosset graph} are also strongly smooth
and prime.  Another class consists of (4) \emph{the 2-connected weakly
  bridged graphs not containing isometric $K_{2,3}^+$ and $R_1$} (by
Theorem \ref{t:wb-smooth-strongly-smooth} they are strongly smooth and they
are prime because they clique complexes are simply connected
\cite{Ch_CAT,ChOs,JaSw}).

Yet another class of prime strongly smooth graphs comes from the
2-dimensional king grid ${\mathbb Z}^2_{\infty}$: (5) \emph{any 2-connected
  subgraph $G$ of ${\mathbb Z}^2_{\infty}$ whose clique complex is simply
  connected} can be shown to be a Helly graph (this follows from the
local-to-global characterization of Helly graphs of \cite{CCHO}), and thus
$G$ is weakly modular. One can directly check that $G$ does not contain the
graphs $W^-_4$, $K_{2,3}$, $K_{2,3}^+$, and $R_1$ as isometric subgraphs,
thus $G$ is smooth. Geometrically, one can view $G$ as the subgraph of
${\mathbb Z}^2_{\infty}$ induced by the vertices lying inside the region of
${\mathbb R}^2$ bounded by a simple non-crossing cycle of
${\mathbb Z}^2_{\infty}$.  However this construction does not work in
higher dimensions: it is well known that each graph can be isometrically
embedded into some king grid ${\mathbb Z}^r_{\infty}$, therefore for enough
large $r$, ${\mathbb Z}^r_{\infty}$ contains the non-smooth graphs $W^-_4$,
$K_{2,3}$, $K_{2,3}^+$, and $R_1$ as isometric subgraphs (already
${\mathbb Z}^4_{\infty}$ contains the first three graphs on 5
vertices). This also shows that Helly graphs (an important class of weakly
modular graphs having simply connected clique complexes) are not smooth or
strongly smooth.

\section*{Acknowledgments}
Victor Chepoi was partially supported by the ANR project MIMETIQUE
(ANR-25-CE48-4089-01). He would like to acknowledge the hospitality of the
Bioinformatics Group during his visit of Leipzig in April, 2026, when this
work was started.

\bibliographystyle{abbrv}
\bibliography{bibliography}

\end{document}